\documentclass{amsart}
\usepackage{etoolbox}
\usepackage[utf8]{inputenc}
\usepackage[T1]{fontenc}
\usepackage{lmodern}
\usepackage[foot]{amsaddr}
\usepackage{fullpage}
\usepackage{microtype}
\usepackage{amsfonts}
\usepackage{graphicx}
\usepackage{tabularx}
\usepackage{array}
\usepackage[usenames,dvipsnames]{color}
\usepackage{amsmath}
\usepackage{amsthm}
\usepackage{amssymb}
\usepackage{xcolor}
\usepackage{tikz}
\usetikzlibrary{backgrounds}
\usetikzlibrary{shapes.geometric}
\usetikzlibrary{calc}
\usepackage{hyperref}
\usepackage{algorithm}
\usepackage{dsfont}
\usepackage{booktabs,makecell,diagbox}

\tikzstyle{legend_general}=[rectangle, rounded corners, thin,
                          top color= white,bottom color=lavander!25,
                          minimum width=2.5cm, minimum height=0.8cm,
                          violet]

\newtheorem{theorem}{Theorem}[section]

\newtheorem{conjecture}[theorem]{Conjecture}

\newtheorem{lemma}[theorem]{Lemma}
\newtheorem{corollary}[theorem]{Corollary}
\theoremstyle{definition}

\def\epsilon{\varepsilon}

\title{Oriented Colouring Graphs of Bounded \\
Degree and Degeneracy}
\author{A. Clow, L. Stacho}
\date{\today}

\begin{document}

\maketitle

\begin{abstract}
This paper considers upper bounds on the oriented chromatic number $\chi_o(G)$, of an oriented graph $G$ in terms of its $2$-dipath chromatic number $\chi_2(G)$, degeneracy $d(G)$, and maximum degree $\Delta(G)$. In particular, we show that for all graphs $G$ with $\chi_2(G) \leq k$ where $k \geq 2$ and $d(G) \leq t$ where $t \geq \log_2(k)$, $\chi_o(G) = 33/10(k t^2 2^t)$. This improves an upper bound of MacGillivray, Raspaud, and Swartz of the form $\chi_o(G) \leq 2^{\chi_2(G)} -1$  to a polynomial upper bound for many classes of graphs, in particular, those with bounded degeneracy. Additionally, we asymptotically improve bounds for the oriented chromatic number in terms of maximum degree and degeneracy. For instance, we show that $\chi_o(G) \leq  (2\ln2 +o(1))\Delta^2 2^\Delta$ for all graphs,  and $\chi_o(G) \leq (2+o(1))\Delta d 2^d$ for graphs where degeneracy grows sublinearly in maximum degree. Here the asypmtotics are in $\Delta$. The former improves the asymptotics of a results by Kostochka, Sopena, and Zhu \cite{kostochka1997acyclic}, while the latter improves the asymptotics of a result by Aravind and Subramanian \cite{aravind2009forbidden}. Both improvements are by a constant factor.
\end{abstract}

\section{Introduction}

An \textit{oriented graph} $G$ is a directed graph whose underlying graph is simple. Throughout this paper every graph we consider is an oriented graph, so it should be understood that stating $G$ is a graph is synonymous with stating $G$ is an oriented graph. We identify parameters of the underlying graph of an oriented graph $G$ as would be normally done given simple graphs and parameters of the orientation of $G$ as is standard with directed graphs. For example $deg(v)$ is the degree of the vertex $v$ independent of orientation whereas $deg^+(v)$ denotes the out-degree, and $deg^-(v)$ the in-degree of the vertex $v$. Similarly, a path $p$ in the graph need not be a directed path. When a path is directed we will call it a dipath.

An \textit{oriented colouring} of a graph $G= (V,E)$ is a proper vertex vertex colouring $c : V \rightarrow \mathbb{N}$ such that if $(u,v), (x,y) \in E$, then 
\begin{itemize}
	\item $c(u) = c(y)$ implies $c(v) \neq c(x)$, and
	\item $c(v) = c(x)$  implies $c(u) \neq c(y)$.
\end{itemize}
If the image of $c$ has cardinality $k$, then we say $G$ has an \textit{oriented $k$-colouring}. For some examples consider Figure~\ref{Fig: OColouringExample}.

\begin{figure}[h!]\label{Fig: OColouringExample}
\begin{center}
    \scalebox{0.75}{
        \begin{tikzpicture}[node distance={15mm}, thick, main/.style = {draw, circle}] 
        
\node[main][fill = blue] (1) at (1,0) {};
	\node[fill=none] at (1,-0.5) (nodes) {$1$};
\node[main][fill = green]  (2) at (0,2) {};
	\node[fill=none] at (0,1.5) (nodes) {$2$};
\node[main][fill = red]  (3) at (2,3) {};
	\node[fill=none] at (2,3.5) (nodes) {$4$};
\node[main][fill = orange]  (4) at (4,2) {};
	\node[fill=none] at (4,1.5) (nodes) {$5$};
\node[main][fill = cyan]  (5) at (3,0) {};
	\node[fill=none] at (3,-0.5) (nodes) {$6$};
\node[main][fill = pink]  (6) at (2,1.5) {};
	\node[fill=none] at (2,1) (nodes) {$3$};
\draw  [->] (2) -- (1);
\draw  [->] (1) -- (5);
\draw  [->] (6) -- (5);
\draw  [->] (4) -- (5);
\draw  [->] (4) -- (3);
\draw  [->] (3) -- (2);
\draw  [->] (6) -- (3);
\draw  [->] (2) -- (6);

\node[main][fill = blue] (7) at (6,0) {};
	\node[fill=none] at (6,-0.5) (nodes) {$1$};
\node[main][fill = green]  (8) at (6,2) {};
	\node[fill=none] at (6,2.5) (nodes) {$2$};
\node[main][fill = green]  (9) at (8,0) {};
	\node[fill=none] at (8,-0.5) (nodes) {$2$};
\node[main][fill = blue]  (10) at (8,2) {};
	\node[fill=none] at (8,2.5) (nodes) {$1$};
\draw  [->] (7) -- (8);
\draw  [->] (10) -- (8);
\draw  [->] (10) -- (9);
\draw  [->] (7) -- (9);

\node[main][fill = blue] (11) at (10,0) {};
	\node[fill=none] at (10,-0.5) (nodes) {$1$};
\node[main][fill = green]  (12) at (10,2) {};
	\node[fill=none] at (10,2.5) (nodes) {$2$};
\node[main][fill = red]  (13) at (12,0) {};
	\node[fill=none] at (12,-0.5) (nodes) {$4$};
\node[main][fill = pink]  (14) at (12,2) {};
	\node[fill=none] at (12,2.5) (nodes) {$3$};
\draw  [->] (11) -- (12);
\draw  [->] (12) -- (14);
\draw  [->] (14) -- (13);
\draw  [->] (13) -- (11);
        \end{tikzpicture}
    }
\end{center}
\caption{Three examples of oriented colourings (consider each component as its own graph).}
\end{figure}
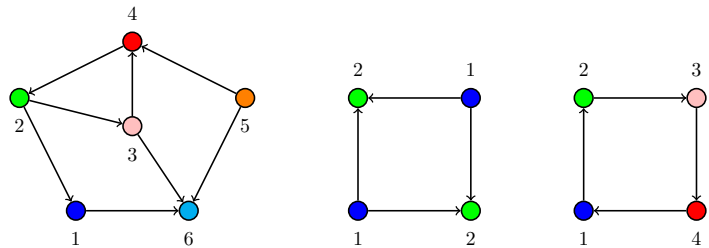

Equivalently, $G$ has an \textit{oriented $k$-colouring} if there exists a graph $H$ of order $k$ and a function $h: V(G)\rightarrow V(H)$ so that for all $(u,v) \in E(G)$, $(h(u),h(v)) \in E(H)$. Such a map $h$ is called an \textit{oriented homomorphism}. See Figure.2 for an example.

\begin{figure}[h!]\label{Fig: OHomomorphismExample}
\begin{center}
    \scalebox{0.85}{
        \begin{tikzpicture}[node distance={15mm}, thick, main/.style = {draw, circle}] 
\node[main][fill= blue] (8) at (6,-1) {}; 
	\node[fill=none] at (5.5,-1) (nodes) {$1$};
\node[main][fill= green] (9) at (6,0.5)  {}; 
	\node[fill=none] at (5.5,0.5) (nodes) {$2$};
\node[main][fill= blue]  (10) at (6,2) {}; 
	\node[fill=none] at (5.5,2) (nodes) {$1$};
\node[main][fill= orange] (11) at (7.5,-1) {}; 
	\node[fill=none] at (8,-1) (nodes) {$5$};
\node[main][fill= red] (12) at (7.5,0.5)  {}; 
	\node[fill=none] at (8,0.5) (nodes) {$4$};
\node[main][fill= pink] (13) at (7.5,2) {}; 
	\node[fill=none] at (8,2) (nodes) {$3$};
\draw  [->] (11) -- (8);
\draw  [->] (12) -- (8);
\draw  [->] (8) -- (13);
\draw  [->] (11) -- (9);
\draw  [->] (9) -- (12);
\draw  [->] (13) -- (9);
\draw  [->] (11) -- (10);
\draw  [->] (12) -- (10);
\draw  [->] (10) -- (13);

\draw[->] (8.25,0.5) -- (9.25,0.5);

\node[main][fill= blue] (1) at (10,1) {}; 
	\node[fill=none] at (9.5,1) (nodes) {$1$};
\node[main][fill= green] (2) at (10,0)  {}; 
	\node[fill=none] at (9.5,0) (nodes) {$2$};
\node[main][fill= pink]  (3) at (12,2) {}; 
	\node[fill=none] at (12.5,2) (nodes) {$3$};
\node[main][fill= orange] (4) at (12,-1) {}; 
	\node[fill=none] at (12.5,-1) (nodes) {$5$};
\node[main][fill= red] (5) at (12,0.5)  {};
	\node[fill=none] at (12.5,0.5) (nodes) {$4$};
\draw  [->] (1) -- (3);
\draw  [->] (4) -- (1);
\draw  [->] (5) -- (1);
\draw  [->] (3) -- (2);
\draw  [->] (4) -- (2);
\draw  [->] (2) -- (5);
        \end{tikzpicture}
    }
\end{center}
\caption{An example of an oriented homomorphism. Colours are depicted as numbers following the definition, as well as actual colours.}
\end{figure}
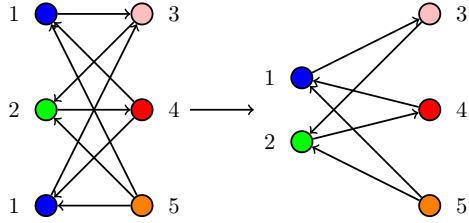

Meanwhile, the \textit{oriented chromatic number} of a graph $G$, denoted $\chi_o (G)$, or simply $\chi_o$ when the choice of $G$ is obvious, is the least integer $k$ so  that $G$ has an oriented $k$-colouring.

This parameter was first introduced and studied by Courcelle \cite{courcelle1994monadic} as a means to encode a graph orientation as a vertex labelling. Since its inception,  $\chi_o(G)$ has been extensively studied with the first major results coming from Raspaud and Sopena~\cite{raspaud1994good}, who proved $\chi_o(G) \leq \chi_a(G) 2^{\chi_a(G) -1}$ where $\chi_a(G)$ is the acyclic chromatic number of the underlying graph of $G$. As Borodin \cite{borodin1979acyclic} had previously shown the acyclic chromatic number of a planar graph is at most $5$, Raspaud and Sopena in fact proved that the oriented chromatic number of a planar graph is at most $80$. A bound that has not been improved in the nearly 30 years since its publication, despite many efforts to do so \cite{sopena2016homomorphisms}. Moreover, it is unknown if there exists a planar graph that requires more than $18$ colours in an oriented colouring \cite{marshall2007homomorphism,marshall2015oriented,sopena2002there,sopena2016homomorphisms}.

Bounding the oriented chromatic number in terms of the maximum degree, $\Delta=\Delta(G)$, was first considered by Sopena \cite{sopena1997chromatic} who showed $\chi_o(G) \leq (2\Delta -1)4^{\Delta-1}$. This was later improved by Kostochka, Sopena, and Zhu~\cite{kostochka1997acyclic} to $\chi_o(G) \leq 2\Delta^2 2^{\Delta}$, later being expanded on by Aravind and Subramanian~\cite{aravind2009forbidden} who showed $\chi_o(G) \leq 16 \Delta d 2^d$ where $d$ is the degeneracy of the graph is question. Recall that degeneracy $d(G)$ or simply $d$ when the choice of $G$ is obvious, is the smallest integer $k$ such that $\delta(H) \leq k$ for all subgraphs $H$ of $G$. Notably, there is no way to eliminate the maximum degree term in the previous upper bound entirely. For example, given any simple graph $G$, the oriented subdivision $H$ of $G$ given by subdividing every edge and orienting the new edges to form $2$-dipaths has $\chi_o(H) \geq \chi(G)$ despite $H$ being $2$ degenerate (see Figure.3). Interestingly, Wood \cite{wood2005acyclic} proved that if $\chi(G) \geq 9$, then $\chi_o(H) = \chi(G)$ where $H$ is the graph obtained from $G$ as before. This fact can be understood in the context of work by Dvo{\v{r}}{\'a}k \cite{dvovrak2008forbidden} who demonstrated that all parameters bounded above and below by the acyclic chromatic number (which the oriented chromatic number is \cite{kostochka1997acyclic, raspaud1994good}) have their values closely related to this type of subdivision.

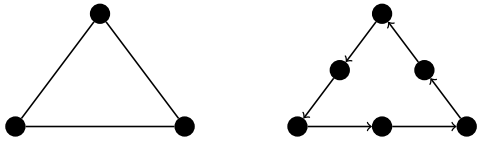
\begin{figure}[h!]\label{Fig: Subdivision}
\begin{center}
    \scalebox{0.75}{
        \begin{tikzpicture}[node distance={15mm}, thick, main/.style = {draw, circle}] 
        
\node[main][fill = black] (1) at (0,0) {};
\node[main][fill = black]  (2) at (3,0) {};
\node[main][fill = black]  (3) at (1.5,2) {};
\draw  (1) -- (2);
\draw  (1) -- (3);
\draw  (2) -- (3);

\node[main][fill = black] (1) at (5,0) {};
\node[main][fill = black] (4) at (6.5,0) {};
\node[main][fill = black]  (2) at (8,0) {};
\node[main][fill = black] (5) at (7.25,1) {};
\node[main][fill = black]  (3) at (6.5,2) {};
\node[main][fill = black] (6) at (5.75,1) {};
\draw  [->] (1) -- (4);
\draw  [->] (4) -- (2);
\draw  [->] (2) -- (5);
\draw  [->] (5) -- (3);
\draw  [->] (3) -- (6);
\draw  [->] (6) -- (1);

        \end{tikzpicture}
    }
\end{center}
\caption{A triangle and a directed subdivision of a triangle where the new edges all form $2$-dipaths.}
\end{figure}

More recently, the maximum degree bound was slightly improved for connected graphs to $\chi_o(G) \leq (\Delta-1)^2 2^{\Delta}+2$ by Das, Nandi,and Sen~\cite{das2017chromatic} in a more general context of connected $(m,n)$-colouring mixed graphs. In the same paper Das, Nandi, and Sen also show that if $G$ ha degeneracy strictly less than its maximum degree, then the plus $2$ term can be dropped implying $\chi_o(G) \leq (\Delta-1)^2 2^{\Delta}$. Attempts to lower this bound for small values of $\Delta$ have also seen some notable progress and remain an active area of research \cite{duffy2019oriented,dybizbanski2020oriented,sopena1996note}.

For a more complete picture of the literature with regard to oriented colouring we recommend Sopena's 2015 survey paper~\cite{sopena2016homomorphisms}.

First proposed by Chen and Wang \cite{min20062}, a \textit{$2$-dipath colouring} of a graph $G= (V,E)$ is a proper vertex colouring $c : V \rightarrow \mathbb{N}$ such that, if $(u,v),(v,w) \in E$, then $c(u) \neq c(w)$. Equivalently  $c : V \rightarrow \mathbb{N}$  is a $2$-dipath colouring if and only if $c$ is a proper colouring of $G^2$, the (directed) square of the graph $G$. Recall that $G$ is oriented and all the edges involved are directed. The $2$-dipath chromatic number of a graph $G$, denoted $\chi_2(G)$ or simply $\chi_2$ when our choice of $G$ is obvious, is the least integer $k$ such that $G$ admits a $2$-dipath $k$-colouring

It should be clear from this definition alone that $\chi_o(G)$ and $\chi_2(G)$ are related parameters. In particular, we can view $\chi_2(G)$ as a localized version of $\chi_o(G)$ as every oriented colouring is a $2$-dipath colouring and $2$-dipath colourings must only satisfy local constraints unlike oriented colourings. Perhaps the best example of this local versus global behaviour is that for all graphs $G$, $\chi_2(G) = \max\{\chi_2(C): C \text{ is a connected component in } G\}$ whereas there exist graphs $H$ such that $\chi_o(H) > \max\{\chi_o(C): C \text{ is a connected component in } H\}$. For an example of such a graph $H$, see Figure.4. This relationship between oriented and $2$-dipath colouring can be understood by a result of MacGillivray and Sherk \cite{macgillivray2014theory}, which characterise if a graph $G$ admits a $2$-dipath $k$ colouring based on the existence of an oriented homomorphism from $G$ to a given target graph. As a result of this, MacGillivray and Sherk \cite{macgillivray2014theory} achieve an upper bound for the oriented chromatic number in terms of the $2$-dipath chromatic number, namely $\chi_o(G) \leq 2^{\chi_2(G)}-1$.

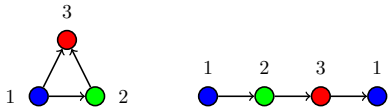
\begin{figure}[h!]\label{Fig: 2Dipath Not Oriented}
\begin{center}
    \scalebox{0.75}{
        \begin{tikzpicture}[node distance={15mm}, thick, main/.style = {draw, circle}] 
        
\node[main][fill = blue] (1) at (0,0) {};
	\node[fill=none] at (-0.5,0) (nodes) {$1$};
\node[main][fill = green]  (2) at (1,0) {};
	\node[fill=none] at (1.5,0) (nodes) {$2$};
\node[main][fill = red]  (3) at (0.5,1) {};
	\node[fill=none] at (0.5,1.5) (nodes) {$3$};
\draw  [->] (1) -- (2);
\draw  [->] (1) -- (3);
\draw  [->] (2) -- (3);

\node[main][fill = blue] (4) at (3,0) {};
	\node[fill=none] at (3,0.5) (nodes) {$1$};
\node[main][fill = green]  (5) at (4,0) {};
	\node[fill=none] at (4,0.5) (nodes) {$2$};
\node[main][fill = red]  (6) at (5,0) {};
	\node[fill=none] at (5,0.5) (nodes) {$3$};
\node[main][fill = blue] (7) at (6,0) {};
	\node[fill=none] at (6,0.5) (nodes) {$1$};
\draw  [->] (4) -- (5);
\draw  [->] (5) -- (6);
\draw  [->] (6) -- (7);
        \end{tikzpicture}
    }
\end{center}
\caption{A digraph with two components is depicted, along with a $2$-dipath colouring which is not an oriented colouring.}
\end{figure}

Our primary contribution is to obtain an improved bound for the oriented chromatic number in terms of the $2$-dipath chromatic number and degeneracy (Theorem~\ref{Thm: Oriented<= 2dipath}). As a corollary (Corollary~\ref{Corollary: Graphs with bounded degeneracy}) we show that in any class of graphs with bounded degeneracy, the orientated chromatic number of a graph $G$ in our class is $O(\chi_2(G)^{2+o(1)})$. Here the coefficient depends on the upper bound for degeneracy in a given class. Notably, our approach was used by the same author in \cite{bradshaw2023injective} to show that $\chi_o(G) = O(k \log k + g)$, where all graphs of genus $g$ have $\chi_2 \leq k$ and $g$ is the Euler genus of $G$. Corollary~\ref{Corollary: Graphs with bounded degeneracy} can be viewed as a generalisation of this result.

We obtain further results by improving the asymptotic bounds for the oriented chromatic number in terms of max degree and degeneracy. In particular, we are able to improve the prior bounds by a constant factor although our results do not apply to graphs with small maximum degree.

The paper is structured as follows. In Section~2 we give some important definitions for the rest of the paper and note some significant results from the literature. In Section~3, we give our improved bounds on the the oriented chromatic number in terms of maximum degree and degeneracy. In Section~4, we obtain our improved upper bound for the oriented chromatic number in terms of the $2$-dipath chromatic number and degeneracy. We conclude with a discussion of future work.

\section{Preliminaries}

We start by introducing some convenient terminology from \cite{aravind2009forbidden, das2017chromatic, kostochka1997acyclic}. Let $G = (V,E)$ be a fixed but arbitrary graph. Let $v \in V $ and let $A \subset N(v)$. Suppose the vertices of $G$ are ordered and let $A = \{u_1, u_2, \dots, u_{|A|}\}$. We define $F(A,v,G) \in \{-1,1\}^{|A|}$, or simply $F(A,v)$ when the choice of $G$ is obvious, to be the vector with the $i$-th entry equal to $1$ if $(v,u_i) \in E$ and equal to $-1$ if $(u_i,v) \in E$. The ordering on the vertices here has no significance beyond allowing us to define $F$.

We say a tournament is $(k,t)$-comprehensive if for all $U \subset V(T)$ where $|U| = k$ and any ${\vec a} \in \{-1,1\}^k$ there exist at least $t$ vertices $z \in V(T)$ where $F(U,z,T) = {\vec a}$. We note that Sopena in \cite{sopena1997chromatic} demonstrated a $(k,1)$-comprehensive graph of order $(k+1)2^k$ while Szekeres and Szekeres in \cite{szekeres1965problem} proved that every $(k,1)$-comprehensive graphs is order at least $(k+2)2^{k-1}-1$. Thus, the matter of determining the order of a smallest $(k,1)$-comprehensive graph is resolved up to a factor of $2$. For an example of a $(2,1)$-comprehensive tournament see Figure.~5.

As we will see however that matter of determining the order of smallest $(k,t)$-comprehensive graphs for $t>1$ is quite useful for oriented colouring and thus deserves its own consideration. Notably we are not the first authors to observe this, as $(k,t)$-comprehensive graphs have featured in many oriented colouring papers to date, although not necessarily by this name.

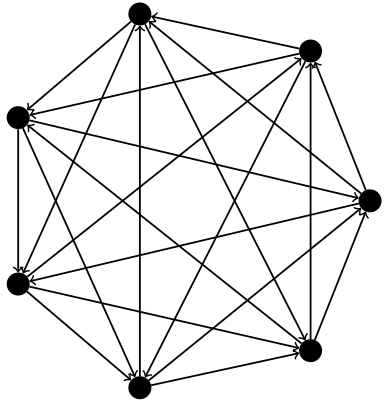
\begin{figure}\label{Fig:G(7;1,2,4)}
\begin{center}
    \scalebox{0.85}{
        \begin{tikzpicture}[node distance={15mm}, thick, main/.style = {draw, circle}] 

\node[main][fill= black] (0) at (2.8,0) {}; 
\node[main][fill= black] (1) at (1.870,2.345) {}; 
\node[main][fill= black] (2) at (-0.8,2.925) {}; 
\node[main][fill= black] (3) at (-2.703,1.302) {}; 
\node[main][fill= black] (4) at (-2.703,-1.302) {}; 
\node[main][fill= black] (5) at (-0.8,-2.925) {}; 
\node[main][fill= black] (6) at (1.870,-2.345) {}; 

\draw  [->] (0) -- (1);
\draw  [->] (1) -- (2);
\draw  [->] (2) -- (3);
\draw  [->] (3) -- (4);
\draw  [->] (4) -- (5);
\draw  [->] (5) -- (6);
\draw  [->] (6) -- (0);

\draw  [->] (0) -- (2);
\draw  [->] (2) -- (4);
\draw  [->] (4) -- (6);
\draw  [->] (6) -- (1);
\draw  [->] (1) -- (3);
\draw  [->] (3) -- (5);
\draw  [->] (5) -- (0);

\draw  [->] (0) -- (4);
\draw  [->] (4) -- (1);
\draw  [->] (1) -- (5);
\draw  [->] (5) -- (2);
\draw  [->] (2) -- (6);
\draw  [->] (6) -- (3);
\draw  [->] (3) -- (0);
        \end{tikzpicture}
    }
\end{center}
\caption{The Cayley graph $T:= Cayley(\mathbb{Z}/7\mathbb{Z};\{1,2,4\})$ which is a smallest $(2,1)$-comprehensive graph.}
\end{figure}

Similarly, it turns out to be convenient to consider well behaved orientations of complete $k$-partite graphs. Specifically, for positive integers $t$ and $k > 1$ an orientations of the complete $k$-partite graph $K:= K_{N,\dots, N} = (P_1,\dots, P_k, E)$ is \textit{$(k,t,N)$-full} if for all $i \in [k]$ and $A \subset \cup_{j\neq i} P_j$ of cardinality at most $t$ and vectors ${\vec a} \in \{-1,1\}^t$, there exists a vertex $v \in P_i$ where 
$$F(A,v,K) = {\vec a}.$$ Observe that a $(k,t,N)$-full graph has exactly $kN$ vertices.

\section{Graphs of Bounded Degree and Degeneracy}

\subsection{Graphs with Bounded Degree}

\begin{lemma}\label{Lemma: Comprehensive}
    If $T$ is a $(k,t)$-comprehensive graph where $k \geq 2$, then $T$ is a $(k-1,2t)$-comprehensive graph.
\end{lemma}

\begin{proof}
    Suppose that $T$ is a $(k,t)$-comprehensive graph where $k\geq 2$. Let $A = \{v_1,\dots, v_{k-1}\} \subset V(T)$ be a fixed but arbitrary and let ${\vec a} = (x_1,\dots, x_{k-1}) \in \{-1,1\}^{k-1}$. Next, let $u \in V(T)\setminus A$ and $B = \{v_1,\dots, v_{k-1},u\}$. As $T$ is $(k,t)$-comprehensive there are at least $t$ vertices $z$ such that $F(B,z,T) = {\vec a}_+$, and at least $t$ vertices $w$ such that $F(B,w,T) = {\vec a}_-$, where ${\vec a}_+ = (x_1,\dots, x_{k-1},1)$ and ${\vec a}_- = (x_1,\dots, x_{k-1},-1)$. For any such $z$ or $w$, $F(A,z,T)=F(A,w,T) = {\vec a}$, so there are at least $2t$ vertices $z$ such that $F(A,z,T)= {\vec a}$. As our choice of $A$ and ${\vec a}$ was arbitrary we conclude that $T$ is $(k-1,2t)$-comprehensive as required.
\end{proof}

\begin{lemma}\label{Lemma: Target (k,t)}
Let $G$ be a graph of maximum degree $\Delta\geq 2$ with degeneracy $d \leq \Delta-1$. If $T$ is a $(\Delta-1, \Delta)$-comprehensive tournament, then $G$ has an oriented homomorphism to $T$.
\end{lemma}

\begin{proof}
Let $G=(V,E)$ and $T$ be as in the statement of the lemma. Let $v_1, v_2, \dots, v_n$ be a fixed degeneracy ordering of $V$. We define $h: V \rightarrow V(T)$ inductively on the degeneracy ordering of $V$ such that it will satisfy 
\begin{enumerate}
	\item $h|_{\{v_1, \dots, v_i\}}$ is a homomorphism from $G[\{v_1, \dots, v_i\}]$ to $T$, and
	\item for all $v_j$ where $j>i$, $h(v_r) \neq h(v_s)$ for all $v_r, v_s \in N(v_j) \cap \{v_1, \dots, v_i\}$.
\end{enumerate}
We can define $h(v_1)$ arbitrarily as it will satisfy (1) and (2) trivially. Suppose $h|_{\{v_1, \dots, v_i\}}$ is already defined and let $A= N(v_{i+1}) \cap \{v_1, \dots, v_i\}$. By definition of degeneracy $|A| \leq \Delta-1$ and by (2), for all distinct $v_r,v_s \in A$, $h(v_r) \neq h(v_s)$. This implies that we need not be concerned about $h(v_r)=h(v_s)$ where $v_r$ and $v_s$ have different orientations to $v_{i+1}$. By our choice of $T$ and Lemma~\ref{Lemma: Comprehensive} there are at least $2^{\Delta-1-|A|}\Delta$ vertices $z \in V(T)$ such that $F(h(A),z,T) = F(A,v_{i+1},T)$. As there are at most $(\Delta-|A|)(\Delta-1) < 2^{\Delta-1-|A|}\Delta$, for any value of $|A| \geq 1$, vertices $v_j$ where $j\leq i$ and $v_j$ has a common neighbour with $v_{i+1}$ in $\{v_{i+2}, \dots, v_n\}$, and as $T$ is order at least $1-\Delta+\Delta^2$ given $T$ is $(1,2^{\Delta-2}\Delta)$-comprehensive by Lemma~\ref{Lemma: Comprehensive}, it follows that there is a vertex $z \in V(T)$ such that for all these $v_j$, $h(v_j) \neq z$ and $F(h(A),z,T) = F(A,v_{i+1},T)$. Choose such a $z$ in $T$ and let  $h(v_{i+1}) = z$. Clearly, $h|_{\{v_1, \dots, v_i, v_{i+1}\}}$ satisfies (1) and (2). As the resulting mapping is a homomorphism, the lemma is proved.
\end{proof}

\begin{lemma}\label{Lemma: Chernoff's Bound}[\cite{motwani1995randomized} Theorem~4.2]
    Let $X$ be a random variable with binomial distribution $Bin(n,p)$ where $0<p<1$ and $0 < \delta \leq 1$. Then, $$\mathbb{P}(X < (1-\delta)\mu) \leq \exp(-\mu\delta^2/2)$$
    where $\mu := \mathbb{E}(X)$.
\end{lemma}

\begin{lemma}\label{Lemma: d<=k-1}
Let $\epsilon>0$ be a fixed constant. There there exist an integer $N$ depending only on $\epsilon$ such that for all $k \geq N$, there exists a $(k-1, k)$-comprehensive tournament of order $\lceil (\ln 2 + \epsilon)k^2 2^k \rceil$.
\end{lemma}

\begin{proof}
Let $T = (V,E)$ be a random orientation of the complete graph on $n$ vertices such that each edge is assigned an orientation uniformly and independently. We will choose the value of $n$ later. Let ${\vec a} \in \{-1,1\}^{k-1}$ and $U \subset V$ such that $|U| = k-1$ be fixed but arbitrary. Let $X_{U,{\vec a}}$ be the random variable which counts the number of vertices $z \in V\setminus U$ such that $F(U,z,T) = {\vec a}$. It follows that the expectation $\mu := \mathbb{E}(X_{U,{\vec a}}) = (n-k+1)2^{1-k}$.

Letting $\delta = 1- \frac{k}{\mu}$, observe that Lemma~\ref{Lemma: Chernoff's Bound} implies 
\begin{align*}
    \mathbb{P}(X_{U,{\vec a}} < (1-\delta)\mu) = \mathbb{P}(X_{U,{\vec a}} < k) \leq \exp(-\mu\delta^2/2) = \exp(k - \frac{\mu}{2}-\frac{k^2}{2\mu}) 
\end{align*}
whenever $0<\delta \leq 1$. 

Now let $n = (\ln 2 + \epsilon)k^2 2^k $. Then, 
\begin{align*}
    \mu = ( (\ln 2 + \epsilon)k^2 2^k  - k +1)2^{1-k} = 2(\ln 2 + \epsilon)k^2 + \frac{1-k}{2^{k-1}}
\end{align*}
implying that for $k \geq 2$, $0< \delta = 1-\frac{k}{\mu} \leq 1$ as required to apply Lemma~\ref{Lemma: Chernoff's Bound}. Next observe that this implies for $k \geq 2$, $$2(\ln 2 + \epsilon)k^2 - 1 \leq \mu \leq 2(\ln 2 + \epsilon)k^2+1.$$
Hence, we can conclude that for $k \geq 2$
\begin{align*}
    \mathbb{P}(X_{U,{\vec a}} < k) \leq \exp(1/2 + k - (\ln 2 + \epsilon)k^2 -\frac{k^2}{4(\ln 2 + \epsilon)k^2+2}) \\
    = \exp((o(1)-\ln2 - \epsilon)k^2)
\end{align*}
where the $o(1)$ is a function of $k$. Applying the union bound over all choices of $U$ and $\vec{a} \in \{-1,1\}^{k-1}$, we arrive at the following inequalities $k \geq 5$,
\begin{align*}
    \mathbb{P}(\exists U,\vec{a}, \textit{ such that } X_{U,\vec{a}}<k+1) \leq \binom{n}{k-1}2^{k-1}\mathbb{P}(X_{U,{\vec a}} < k+1)\\
    \leq  \binom{n}{k-1}2^{k-1}\exp(-(\ln2 + \epsilon -o(1))k^2)\\
    \leq n^k \exp((o(1)-\ln2 - \epsilon)k^2) \\
    \leq  ((\ln 2 + \epsilon)k^2 2^k+1)^{k} \exp((o(1)-\ln2 - \epsilon)k^2) \\
    = \exp((\ln2+o(1))k^2+(o(1)-\ln2 - \epsilon)k^2)\\
    = \exp((o(1)-\epsilon)k^2)\rightarrow 0
\end{align*}
as $k \rightarrow \infty$, given the $o(1)$ is again in terms of $k$. Furthermore, note that all the functions captured by the $o(1)$ are monotone when $k \geq 5$, so we conclude that there exists an integer $N$ depending on $\epsilon$ such that for all $k \geq N$, $\mathbb{P}(\exists U,\vec{a}, \textit{ such that } X_{U,\vec{a}}<k)< 1$. 

Therefore, there exists an $N$ such that for all $k \geq N$, there is a positive probability that $T$ is $(k-1, k)$-comprehensive. Hence, for all $k \geq N$ there must exist a $(k-1, k)$-comprehensive tournament as desired.
\end{proof}

\begin{theorem}\label{Thm: Oriented <= MaxDegree}
Let $\epsilon>0$ be a fixed constant. There exists an integer $N$ depending only on $\epsilon$ such that for all $k \geq N$
\begin{enumerate}
    \item if $G$ is a graph with $d(G) < k$ and $\Delta(G) \leq k$, then $\chi_o(G) \leq \lceil (\ln2 +\epsilon)k^2 2^k \rceil$, and
    \item if $G$ is a connected graph with $\Delta(G) \leq k$, then $\chi_o(G) \leq \lceil (\ln2 +\epsilon)k^2 2^k \rceil+2$, and 
    \item if $G$ is a graph with $\Delta(G) \leq k$, $\chi_o(G) \leq 2\lceil (\ln2 +\epsilon)(k+1)^2 2^k\rceil$.
\end{enumerate}
\end{theorem}

\begin{proof}
Let $\epsilon>0$ be a fixed constant and let $N$ be the least integer that guarantees for all $k \geq N$, there exists a $(k-1,k+1)$-comprehensive tournament of order $\lceil (\ln2 +\epsilon)k^2 2^k\rceil$. Note that such an $N$ exists by Lemma~\ref{Lemma: d<=k-1}.

\underline{Case.1:} $G$ is a graph with $d(G) < k$ and $\Delta(G) \leq k$. Then Lemma~\ref{Lemma: Target (k,t)} implies that $G$ has an oriented homomorphism to all $(k-1,k+1)$-comprehensive tournaments $T$. By the choice of $k\geq N$, there exists a  $(k-1,k+1)$-comprehensive tournament of order $\lceil (\ln2 +\epsilon)k^2 2^k\rceil$. Hence, $\chi_o(G) \leq \lceil (\ln2 +\epsilon)k^2 2^k\rceil$ as desired.

\underline{Case.2:} $G$ is a connected graph with $\Delta(G) \leq D$. Then either $d(G) < k-1$, or $d(G) = \Delta(G) = k$ . In the former case the result follows by Case.1. Suppose then that $G$ is $d(G) = \Delta(G) = k$. It is well known that a connected graph has $d= \Delta$ if and only if it is $\Delta$-regular. Thus, $G$ is $\Delta$-regular. Let $e = (u,v) \in E$ be fixed but arbitrary and let $H = G-e$. Then, $d(H) < \Delta(H) = \Delta(G) = k$. By the argument in Case 1, there is an oriented colouring $\phi_0: V(H) \rightarrow \{1,\dots ,\lceil (\ln2 +\epsilon)k^2 2^k \rceil\}$ of $H$. Let $\phi$ be a vertex colouring of $G$ such that for all $w \neq u,v$, $\phi(w) = \phi_0(w)$, $\phi(u) = \lceil (\ln2 +\epsilon)k^2 2^k \rceil+1$, and $\phi(v) = \lceil (\ln2 +\epsilon)k^2 2^k \rceil +2$. As $\phi_0$ is an oriented colouring it is easy to verify that $\phi$ is also an oriented $\lceil (\ln2 +\epsilon)k^2 2^k \rceil+2$-colouring. Hence,  $\chi_o(G) \leq \lceil (\ln2 +\epsilon)k^2 2^k\rceil+2$ as desired.

\underline{Case.3:} $G$ is a graph with $\Delta(G) \leq k$. If $d(G) < k$, then we are in Case.1, thus, $d(G) = \Delta(G) = k$. Now, by Lemma~\ref{Lemma: d<=k-1} there exists a $(k,k+2)$-comprehensive tournament of order $2\lceil (\ln2 +\epsilon)(k+1)^2 2^k\rceil$. Then by Lemma~\ref{Lemma: (k,t)-full target} $G$ has an oriented homomorphism to this tournament, which implies $\chi_o(G) \leq 2\lceil (\ln2 +\epsilon)(k+1)^2 2^k\rceil$ as desired.
 \end{proof}

We note that by examining the bounds in Lemma~\ref{Lemma: d<=k-1} it can be observed that letting $k \geq 22$ is sufficient to let $\epsilon = 3/10$ which improves the known coefficients from \cite{das2017chromatic,kostochka1997acyclic} of $1$ (for cases 1 and 2) and $2$ (for case 3). However this improvement is somewhat marginal, so it is natural to ask how large $k$ must be to reach a more significant improvement. Unfortunately the functions from Lemma~\ref{Lemma: d<=k-1} grow so fast that verifying this is non-trivial. See Table~\ref{Table: S2.1 coefficient} for a short list of smallest values $k$ where we may apply a given natural choices of $\epsilon$, which we verified using a computer.

\vspace{0.5cm}
\begin{table}[!h]\label{Table: S2.1 coefficient}
\centering
\begin{tabular}{| c | c |} 
\hline 
\makecell{\textit{Coefficient}} & \makecell{$k \geq$}\\ \hline
$\ln2 + 1$ & $4$ \\ \hline
$\ln2 + 1/2$ & $11$ \\ \hline
$\ln2 + 2/5$ & $15$ \\ \hline
$\ln2 + 3/10$ & $22$ \\ \hline
$\ln2 + 11/40$ & $25$ \\ \hline
$\ln2 + 1/4$ & $28$ \\ \hline
\hline 
\end{tabular}
\caption{Coefficients for Theorem~\ref{Thm: Oriented <= MaxDegree} in the first column with the smallest $k$ such that these coefficients can be used appearing in the second column.}
\end{table}

\subsection{Graphs with Bounded Degeneracy}

\begin{lemma}\label{Lemma: Target (k,t)2}
Let $G$ be a graph of maximum degree $\Delta$ and degeneracy $d$. If $T$ is a $(d, d\Delta)$-comprehensive tournament, then $G$ has an oriented homomorphism to $T$.
\end{lemma}

\begin{proof}
Let $G=(V,E)$ and $T$ be as in the statement of the lemma. Let $v_1, v_2, \dots, v_n$ be a fixed degeneracy ordering of $V$. We define $h: V \rightarrow V(T)$ inductively on the degeneracy ordering of $V$ such that it will satisfy 
\begin{enumerate}
	\item $h|_{\{v_1, \dots, v_i\}}$ is a homomorphism from $G[\{v_1, \dots, v_i\}]$ to $T$, and
	\item for all $v_j$ where $j>i$, $h(v_r) \neq h(v_s)$ for all $v_r, v_s \in N(v_j) \cap \{v_1, \dots, v_i\}$.
\end{enumerate}
We can define $h(v_1)$ arbitrarily as it will satisfy (1) and (2) trivially. Suppose $h|_{\{v_1, \dots, v_i\}}$ is already defined and let $A= N(v_{i+1}) \cap \{v_1, \dots, v_i\}$. By definition of degeneracy $|A| \leq d$ and by (2), for all distinct $v_r,v_s \in A$, $h(v_r) \neq h(v_s)$. This implies that we need not be concerned about $h(v_r)=h(v_s)$ where $v_r$ and $v_s$ have different orientations to $v_{i+1}$. By our choice of $T$ there are at least $d\Delta $ vertices $z \in V(T)$ such that $F(h(A),z,T) = F(A,v_{i+1},T)$. As there are at most $(\Delta-|A|)(d-1) < d\Delta$ vertices $v_j$ where $j\leq i$ and $v_j$ has a common neighbour with $v_{i+1}$ in $\{v_{i+2}, \dots, v_n\}$, and as $T$ is order at least $1-\Delta+d\Delta$ given $T$ is $(1,2^{d-1}d\Delta)$-comprehensive by Lemma~\ref{Lemma: Comprehensive}, it follows that there is a vertex $z \in V(T)$ such that for all these $v_j$, $h(v_j) \neq z$ and $F(h(A),z,T) = F(A,v_{i+1},T)$. Choose such a $z$ in $T$ and let  $h(v_{i+1}) = z$. Clearly, $h|_{\{v_1, \dots, v_i, v_{i+1}\}}$ satisfies (1) and (2). As the resulting mapping is a homomorphism, the lemma is proved.
\end{proof}

\begin{lemma}\label{Lemma: (d,kd)-comprehensive}
Let  $\alpha: \mathbb{N} \rightarrow (0,1]$ be a monotone function such that $\alpha(k)k^2 \rightarrow \infty$ as $k\rightarrow \infty$. There exists an integer $N$ depending on $\alpha$ such that for all $k \geq N$, there is a $(\alpha(k)k, \alpha(k) k^2)$-comprehensive tournament of order $\lceil (2\alpha(k)\ln2+2) \alpha(k) k^2 2^{\alpha(k) k} \rceil$.
\end{lemma}

\begin{proof}
Let $\alpha: \mathbb{N} \rightarrow \mathbb{R}^+$ be a bounded monotone function such that $\alpha(k)k^2 \rightarrow \infty$ as $k\rightarrow \infty$. As in Lemma~\ref{Lemma: d<=k-1} we suppose $T$ is a random tournament, let  $U\subseteq V$ with $|U| = \alpha(k) k$ be fixed but arbitrary. Let $X_{U,{\vec a}}$ be the random variable which counts the number of vertices $z \in V\setminus U$ such that $F(U,z,T) = {\vec a}$ where ${\vec a} \in \{-1,1\}^{\alpha(k) k}$. Then $\mu := \mathbb{E}(X_{U,{\vec a}}) = (n-\alpha(k) k)2^{-\alpha(k) k}$.

Applying Lemma~\ref{Lemma: Chernoff's Bound} with $\delta = 1 - \frac{\alpha(k) k^2}{\mu}$ we see that 
\begin{align*}
\mathbb{P}(X_{U,{\vec a}} < (1-\delta)\mu) = \mathbb{P}(X_{U,{\vec a}} < \alpha(k) k^2) \leq \exp(-\mu\delta^2/2)\\
= \exp(\alpha(k)k^2 - \frac{\alpha^2(k)k^4}{2\mu} - \frac{\mu}{2})
\end{align*}
whenever $0<\delta \leq 1$. Let $n = (2\alpha(k)\ln2+2)\alpha(k) k^2 2^{\alpha(k) k}$, then 
$$
\mu =  (2\alpha(k)\ln2+2)\alpha(k) k^2 - \frac{\alpha(k) k}{2^{\alpha(k) k}}
$$
which, given our assumption that $\alpha$ is monotone and $\alpha(k)k^2 \rightarrow \infty$ as $k\rightarrow \infty$, implies that for large enough $k$ depending on $\alpha$, we have $0 < \delta \leq 1$ as required by Lemma~\ref{Lemma: Chernoff's Bound}. Also observe that for large enough $k$, depending on $\alpha$,
$$
(2\alpha(k)\ln2+2) \alpha(k) k^2 - 1 \leq \mu \leq (2\alpha(k)\ln2+2) \alpha(k) k^2 +1.
$$
Hence, for large enough $k$ depending on $\alpha$,
\begin{align*}
\mathbb{P}(X_{U,{\vec a}} < \alpha(k) k^2 +1) \leq \exp(\alpha(k)k^2 - \frac{\alpha^2(k)k^4}{2\mu} - \frac{\mu}{2}) \\
\leq \exp(\frac{1}{2}+\alpha(k)k^2 - \frac{\alpha^2(k)k^4}{2(2\alpha(k)\ln2+2) \alpha(k) k^2 +2} - \frac{(2\alpha(k)\ln2+2) \alpha(k) k^2}{2})\\
= \exp((\frac{2-(2\alpha(k)\ln2+2)}{2}-\frac{1}{4\alpha(k)\ln2+4} +o(1))\alpha(k)k^2) \\
\leq \exp((o(1) - \alpha(k)\ln2-\frac{1}{4\ln2+4})\alpha(k)k^2)
\end{align*}
where the asymptotics are in $k$. Now applying the union bound and choosing $k$ to be sufficiently large with respect to $\alpha$,
\begin{align*}
\mathbb{P}(\exists U, {\vec a},\text{ such that }X_{U,{\vec a}} < \alpha(k) k^2 + 1) \leq \dbinom{n}{\alpha(k) k} 2^{\alpha(k) k} \mathbb{P}(X_{U,{\vec a}} < \alpha(k) k^2) \\
\leq \frac{n^{\alpha(k) k}}{\alpha(k) k!} 2^{\alpha(k) k} \exp((o(1) - \ln2\alpha(k)-\frac{1}{4\ln2+4})\alpha(k) k^2) \\
\leq  ((2\alpha(k)\ln2+2) \alpha(k) k^2 2^{\alpha(k) k})^{\alpha(k)k}\exp((o(1) - \ln2\alpha(k)-\frac{1}{4\ln2+4})\alpha(k) k^2) \\
\leq 2^{(1+o(1))\alpha^2(k)k^2}\exp((o(1) - \ln2\alpha(k)-\frac{1}{4\ln2+4})\alpha(k) k^2)\\
= \exp((o(1) - \ln2\alpha(k)-\frac{1}{4\ln2+4})\alpha(k) k^2+(\alpha(k)\ln2+o(1))\alpha(k)k^2) \\
= \exp((o(1)-\frac{1}{4\ln2+4})\alpha(k)k^2)
\end{align*}
recalling that $\alpha(k)k^2\rightarrow \infty$ as $k \rightarrow \infty$ we note that $\mathbb{P}(\exists U, {\vec a},\text{ such that }X_{U,{\vec a}}) \rightarrow 0$ as $k\rightarrow \infty$. Given $\mathbb{P}(\exists U, {\vec a},\text{ such that }X_{U,{\vec a}}) \rightarrow 0$, we conclude there is an integer $N$ such that for all $k \geq N$, $\mathbb{P}(\exists U, {\vec a},\text{ such that }X_{U,{\vec a}}) < 1$. This implies for all for all $k \geq N$, there is a $(\alpha(k)k, \alpha(k) k^2)$-comprehensive tournament of order $\lceil (2\alpha(k)\ln2+2) \alpha(k) k^2 2^{\alpha(k) k} \rceil$ as desired.
\end{proof}

Note that in the statement of the lemma we allow $\alpha = o(1)$.

\begin{theorem}\label{Thm: d<< Delta}
Let  $\alpha: \mathbb{N} \rightarrow (0,1]$ be a monotone function such that $\alpha(k)k^2 \rightarrow \infty$ as $k\rightarrow \infty$. There exists an integer $N$ depending on $\alpha$ such that for all $k \geq N$, if $G$ is a graph with $\Delta(G) \leq k$ and $d(G) \leq \alpha(k)k$, then $\chi_o(G) \leq \lceil (2\alpha(k)\ln2+2) \alpha(k) k^2 2^{\alpha(k) k} \rceil$.
\end{theorem}

\begin{proof}
Let  $\alpha: \mathbb{N} \rightarrow (0,1]$ be a monotone function such that $\alpha(k)k^2 \rightarrow \infty$ as $k\rightarrow \infty$.By Lemma~\ref{Lemma: (d,kd)-comprehensive} there exists a $(\alpha(k)k, \alpha(k) k^2)$-comprehensive tournament $T$ of order $\lceil (2\alpha(k)\ln2+2) \alpha(k) k^2 2^{\alpha(k) k} \rceil$. If $G$ is a graph with $\Delta(G) \leq k$ and $d(G) \leq \alpha(k)k$, then Lemma~\ref{Lemma: Target (k,t)2} implies that $G$ has a homomorphism to $T$. Thus, $\chi_o(G) \leq \lceil (2\alpha(k)\ln2+2) \alpha(k) k^2 2^{\alpha(k) k} \rceil$ as desired.
\end{proof}

\begin{corollary}\label{Corollary: Asyptotics of d<< D}
If $\{G_n\}$ is a sequence of graphs satisfying that $d(G_n) = o(\Delta(G_n))$, then
$$\chi_o(G_n) \leq (2+o(1))\Delta d 2^{d}$$ where $d= d(G_n)$, and $\Delta=\Delta(G_n)$, and the asymptotics are in $\Delta$.
\end{corollary}

Notice that in the statement of the corollary given we assume that $d(G_n) = o(\Delta(G_n))$ it is implicit that $\Delta(G_n) \rightarrow \infty$. For the sake of any readers less familiar with asymptotic notation, Corollary~\ref{Corollary: Asyptotics of d<< D} is stating that given any graph sequence as required, for all $\epsilon>0$ there exists a large enough integer $N$ such that for all $n \geq N$, $\chi_o(G_n) \leq (2+\epsilon)\Delta(G_n) d(G_n) 2^{d(G_n)}$. In this way, for any sequence or family of graphs where the degeneracy is sublinear in max degree,  for any $\epsilon>0$, the coefficient from Theorem~\ref{Thm: d<< Delta} may be improved to $2+\epsilon$ in the limit of $\Delta \rightarrow \infty$.

\vspace{0.5cm}
\begin{table}[!h]\label{Table: d<< Delta}
\centering
\begin{tabular}{| c | c | c |} 
\hline 
\makecell{$d \leq $} & \makecell{ $\chi_o(G) \leq $} & \makecell{ $\chi_o(G) = O(-)$}\\ \hline
$\alpha \Delta$ &  $\alpha(2\alpha \ln 2 + 2) \Delta^2 2^{\alpha\Delta}$ & $O(\Delta^2 2^{\alpha \Delta})$\\ \hline
$\Delta^\alpha$ &  $(2+o(1))\Delta^{1+\alpha} 2^{\Delta^\alpha}$ & $O(\Delta^{1+\alpha} 2^{\Delta^\alpha})$ \\ \hline
$\log_{2} \Delta$ &  $(2+o(1))\Delta^2 \log_2{\Delta}$ & $O(\Delta^2 \log{\Delta})$\\ \hline
$c$ &  $(2c2^c + o(1)) \Delta$ & $O(\Delta)$ \\
\hline 
\end{tabular}
\caption{Asymptotic bounds from Theorem~\ref{Thm: d<< Delta} when $d \ll \Delta$. Note that in this example $c$ is a constant.}
\end{table}

\section{Bounds on Oriented Chromatic Number in $2$-Dipath Chromatic Number.}

\begin{lemma}\label{Lemma: (k,t)-full target}
Let $K$ be a $(k,d,N)$-full graph. If $G$ is a graph with $\chi_2(G) \leq k$ and degeneracy $d(G) \leq d$, then $G$ has an oriented homomorphism to $K$.
\end{lemma}

\begin{proof}
    Let $G=(V,E)$ be a graph with $\chi_2(G) = k$ and degeneracy $d(G) = d$ and let $K$ is a $(k,d,N)$-full graph. Let $\phi:V \rightarrow \{1,\dots, k\}$ be a $2$-dipath colouring of $G$. We will build an oriented homomorphism $h:G \rightarrow K$ satisfying $h(u) \in P_i$ if and only if $i = \phi(u)$. 

    Let $v_1,v_2, \dots, v_n$ be a degeneracy ordering of $G$, that is $|N(v_j) \cap \{v_1,\dots, v_{j-1}\}| \leq d$ for all $1 \leq j\leq n$. Suppose we have defined $h(v_j)$ for all $j < s$ satisfying our assumed condition that $h(u) \in P_i$ if and only if $i = \phi(u)$. Let $A = N(v_s) \cap \{v_1,\dots, v_{s-1}\}$. By our choice of the degenerate ordering, $|A|\leq d$. Now consider $h(A)$ the image of $A$ under $h$. Notice that $|h(A)| \leq |A| \leq d$ and $h(u) \notin P_{\phi(v_s)}$ for all $u \in A$ given $\phi(u) \neq \phi(v_s)$ as $\phi$ is a proper colouring. Let ${\vec a} \in \{-1,1\}^{|A|}$ such that $F(A,v_s,G) = {\vec a}$. Then let ${\vec b} \in  \{-1,1\}^{|h(A)|}$ be defined by ${\vec b}(h(u)) = {\vec a}(u)$. Note that as $\phi$ is a $2$-dipath colouring of $G$, if $h(u) = h(w)$ for $u,w \in A$, then in $G$ both edges $uv_s$ and $wv_s$ are both oriented either from or towards $v_s$. So $\vec{b}$ is well defined.

    Given $K$ is $(k,d, N)$-full, there exists $x \in P_{\phi(v_s)}$ such that $F(h(A),x,K) = {\vec b}$. Let $h(v_s) = x$. It is clear that $h$ is a homomorphism as required.
\end{proof}

\begin{lemma}\label{Lemma: (k,t)-full}
If $k \geq 2$ and $t \geq \log_2{k}$, then there exists a $(k,t, \frac{33}{10} t^{2} 2^{t})$-full graph.
\end{lemma}

\begin{proof}
Let $k \geq 2$ and $t \geq \log_2{k}$. Notice that this implies $k \leq 2^t$. Consider a random orientation of the complete $k$-partite graph $K = K_{N,\dots, N} = (P_1,\dots, P_k,E)$ where $N = \frac{33}{10} t^{2} 2^{t}$. Note each edge of $K$ is oriented independently and uniformly. For each fixed value $i \in \{1,2,\dots, k\}$ and a subset $A\subset \cup_{j\neq i} P_j$ satisfying $|A|=t$, let $X_{i,A}$ be the random variable 
\[
X_{i,A} := \sum_{{\vec a} \in \{-1,1\}^t} \mathds{1}_{\forall v \in P_i, F(A,v,K)\neq {\vec a}}.
\]
That is, $X_{i,A}$ counts the number of vectors ${\vec a}  \in \{-1,1\}^k$ such that no vertex in $P_i$ has orientation ${\vec a}$ with respect to $A$. Observe that the orientation of two edges between a vertex of $A$ and two distinct vertices of $P_i$ is independent. Furthermore, observe that this implies $X_{i,A} = 0$ is equivalent to a random function from a domain of size $N$ to a codomain of size $2^t$ being surjective. Hence, 
\[
\mathbb{P}(X_{i,A}>0) \leq 2^t (1-2^{-t})^N \leq 2^t e^{-2^{-t}N} = 2^t e^{-\frac{33}{10}t^2}
\]
Applying the union bound,
\begin{center}
\begin{align*}
\mathbb{P}(\exists i, A, \text{ such that }  X_{i,A} > 0) \leq k \dbinom{(k-1)N}{t} \mathbb{P}(X_{i,A}>0) \\
\leq k\dbinom{(k-1)N}{t} 2^t e^{-\frac{33}{10}t^2} \\
\leq k \frac{(kN)^t}{t!} 2^t  e^{-\frac{33}{10}t^2}\\
\leq 2k^{t+1}N^t e^{-\frac{33}{10}t^2}\\
\end{align*}
\end{center}
applying a logarithm to simplify the computation gives
\begin{center}
\begin{align*}
1+(t+1)\log_2{k} + t\log_2(N) - \frac{33}{10}\log_2(e)t^2 \\
\leq 1+(t+1)\log_2{2^t} + t\log_2(\frac{33}{10}t^2 2^t) - \frac{33}{10}\log_2(e)t^2\\
\leq 1 + (1+\log_2{\frac{33}{10}})t + 2\log_2(t)t+ (1-\frac{33}{10}\log_2(e))t^2 < 0
\end{align*}
\end{center}
for all $t \geq 1$. We note that $t \geq \log_2{k} \geq \log_2{2} = 1$ given $k \geq 2$ implies that $t \geq 1$. As this is a logarithm of an upper bound on $\mathbb{P}(\exists i, A, \text{ such that }  X_{i,A} > 0)$, it follows for all $k \geq 2$, and $t \geq \log_2{k}$, $\mathbb{P}(\exists i, A, \text{ such that }  X_{i,A} > 0)< 1$. Thus, for all $k \geq 2$  $t \geq \log_2{k}$ there exists a $(k,t, \frac{27}{10} t^{2} 2^{t})$-full graph.
\end{proof}

We note that the coefficient of $\frac{33}{10}$ in Lemma~\ref{Lemma: (k,t)-full} can be improved to $\frac{1}{\log_2(e)}+ \epsilon$, for any $\epsilon>0$, if we alter the statement to suppose that $k$ (and therefore $t$) are sufficiently large. However, doing so would mean that Lemma~\ref{Lemma: (k,t)-full} could not be applied in proving Theorem~\ref{Thm: Oriented<= 2dipath} unless a similar assumption about $k$ and $t$ being large is made. Thus, our result would not apply to all graphs. For example taking the coefficient to be $1$ rather than $33/10$ we must choose $t \geq 23$ rather than $t \geq 1$, which would force $k\geq 2^{23}$. We chose to set our coefficient at $33/10$ as it the smallest ``nice" fraction so that Theorem~\ref{Thm: Oriented<= 2dipath} applies to all graphs with at least $1$ edge. See Table~\ref{Table: O <= 2-Di} for a longer list of possible improvements to the coefficient in Lemma~\ref{Lemma: (k,t)-full} and Theorem~\ref{Thm: Oriented<= 2dipath}, as well as an indication of the smallest $k$ and $t$ where these coefficients could be applied.

\begin{theorem}\label{Thm: Oriented<= 2dipath}
Let $k\geq 2$ be a fixed but arbitrary integer. Then for all integers $t \geq \log_2{k}$ and for all graphs $G$ with $d(G) \leq t$ and $\chi_2(G) \leq k$, $$\chi_o(G) \leq \frac{33}{10} k t^{2} 2^{t}.$$
\end{theorem}

\begin{proof}
Let $k\geq 2$ and $t \geq \log_2{k}$ be fixed but arbitrary integers and let $G$ be a graph with $d(G) \leq t$ and $\chi_2(G) \leq k$. By Lemma~\ref{Lemma: (k,t)-full} there exists a $(k,t, \frac{33}{10} t^{2} 2^{t})$-full graph $K$. Then Lemma~\ref{Lemma: (k,t)-full target} implies $G$ has an oriented homomorphism to $K$. Thus, $\chi_o(G) \leq |V(K)| = \frac{33}{10} k t^{2} 2^{t}$ as desired.
\end{proof}

\vspace{0.5cm}
\begin{table}[!h]\label{Table: O <= 2-Di}
\centering
\begin{tabular}{| c | c | c |} 
\hline 
\makecell{\textit{Coefficient}} & \makecell{$k \geq$} & \makecell{$t \geq$} \\ \hline
$33/10$ &  $2$ & $1$\\ \hline
$3$ &  $4$ & $2$\\ \hline
$5/2$ &  $4$ & $2$\\ \hline
$2$ &  $8$ & $3$\\ \hline
$3/2$ &  $2^6$ & $6$\\ \hline
$1$ &  $2^{23}$ & $23$\\ \hline
$3/4$ &  $2^{193}$ & $193$\\ \hline
$7/10$ &  $2^{2310}$ & $2310$\\ \hline
$139/200$ &  $2^{10135}$ & $10135$\\ \hline
\hline 
\end{tabular}
\caption{Improved coefficients for Theorem~\ref{Thm: Oriented<= 2dipath}, the smallest $k$ such that they apply, and the smallest $t$ where this applies given the smallest $k$ column.}
\end{table}

Significantly we can obtained a particularly nice corollary if the $2$-dipath chromatic number of a class is not bounded, but the class has bounded degeneracy. This is because for graphs in such a class we can take $t = \max\{\log_2(\chi_2(G)),d\}$ where $d$ is an upper bound on the degeneracy of every graph in our class, which gives an asymptotic upper bound of $(\frac{33}{10}+o(1)) k^2 \log^2_2(k)$ for the oriented chromatic number of such graphs $G$. Recall that such classes exist as demonstrated in the introduction. Also recall such classes are notable given the result of Dvo{\v{r}}{\'a}k in \cite{dvovrak2008forbidden}.

\begin{corollary}\label{Corollary: Graphs with bounded degeneracy}
If $\mathcal{G}$ is a family of graphs with bounded degeneracy,  then for $G \in \mathcal{G}$, $\chi_o(G)  = O(\chi_2(G)^{2+o(1)})$ where the asymptotics are in $\chi_2(G)$. 
\end{corollary}

\section{Future Work}

In this paper we give improved bounds for the oriented chromatic number in terms of $2$-dipath chromatic number, maximum degree, and degeneracy. In both cases, we colour vertices inductively on a degeneracy ordering, where the existence of a colour for the next vertex is guaranteed by a special property of a large graphs which acts as a universal target. Hence, our arguments hinge on establishing the existence of such graphs (either $(k,t,N)$-full or $(k,t)$-comprehensive) of a given order, which we prove using the first moment method. Some natural questions immediately follow from this. 

Perhaps most obvious is, how close are the upper bounds of this paper from being tight? That is, prove or disprove the existence of graphs $G$ with oriented chromatic number near our bounds. At time of writing little is known about this in the literature. The notable exception to this being result from \cite{kostochka1997acyclic} which shows that there are graphs of maximum degree $\Delta$ and oriented chromatic number approximately $ 2^{\Delta/2}$. As a result it seems that progress on lower bounds of this kind would require new and creative ideas that might provide new insights into the oriented chromatic number.

Additionally, Theorem~\ref{Thm: Oriented<= 2dipath} opens the door to similar questions in the perhaps easier to understand context  of well behaved graph class rather than general graphs. In particular families with bounded degeneracy seem good candidate classes given Corollary~\ref{Corollary: Graphs with bounded degeneracy}. To this end we conjecture that Corollary~\ref{Corollary: Graphs with bounded degeneracy} is tight up to a subpolyonomial factor.

\begin{conjecture}
    There exists a class of graphs $\mathcal{G}$ with bounded degeneracy, such that there exists a sequence of graphs  $\{G_n\}_{n\geq 1}$ in $\mathcal{G}$ with $\chi_2(G_n) \rightarrow \infty$ and $\chi_o(G_n) = \Omega(\chi_2(G_n)^2)$ as $n\rightarrow \infty$.
\end{conjecture}

\section*{Acknowledgements}
We would like to thank Dr.Peter Bradshaw (University of Illinois Urbana-Champaign) for providing feedback during the process of drafting this paper. We would also like to acknowledge the support of the Natural Sciences and Engineering Research Council of Canada (NSERC) for support through the Canadian Graduate Scholarship - Master program, and supported in part by the NSERC Discovery Grant R611368.

\bibliographystyle{plain}
\bibliography{OrientedColoring}

\end{document}